\documentclass[12pt]{amsart}
\usepackage{amsmath,amssymb,amsthm}
\usepackage{latexsym}            
\usepackage{hyperref}
\usepackage{multirow}
\usepackage{color,colordvi,graphicx}
\numberwithin{equation}{section}
\newtheorem{theorem}{Theorem}
\newtheorem{lemma}[theorem]{Lemma}

\theoremstyle{definition}

\newtheorem{example}[theorem]{Example}

\newtheorem{remark}[theorem]{Remark}

\newcounter{FNC}[page]
\def\fauxfootnote#1{{\addtocounter{FNC}{2}$^\fnsymbol{FNC}$%
     \let\thefootnote\relax\footnotetext{$^\fnsymbol{FNC}$#1}}}
\newcommand{\sA}{{\mathcal A}}
\newcommand{\sB}{{\mathcal B}}
\newcommand{\sE}{{\mathcal E}}
\newcommand{\sED}{\mathcal{ED}}
\newcommand{\sEH}{\mathcal{EH}}
\newcommand{\sF}{{\mathcal F}}

\newcommand{\bN}{{\mathbb N}}
\newcommand{\bZ}{{\mathbb Z}}
\newcommand{\bR}{{\mathbb R}}
\newcommand{\bC}{{\mathbb C}}
\newcommand{\bP}{{\mathbb P}}
\newcommand{\bQ}{{\mathbb Q}}

\newcommand{\sH}{{\mathcal H}}
\newcommand{\sL}{{\mathcal L}}
\newcommand{\sLH}{\mathcal{LH}}
\newcommand{\sM}{{\mathfrak M}}

\newcommand{\sN}{{\mathcal N}}

\newcommand{\Var}{{\mathcal V}}

\newcommand{\conv}{{\rm conv}}

\newcommand{\three}[3]{\left(\begin{smallmatrix}#1\\#2\\#3\end{smallmatrix}\right)}

\newcommand{\defcolor}[1]{\Blue{#1}}
\newcommand{\demph}[1]{\defcolor{{\sl #1}}}

\title[Newton polytopes and witness sets]{Newton polytopes and witness sets}

\author{Jonathan D. Hauenstein}
\address{Jonathan D. Hauenstein\\
         Department of Mathematics\\
         North Carolina State University\\
         Raleigh\\
         North Carolina \ 27695\\
         USA}
\email{hauenstein@ncsu.edu}
\urladdr{\url{http://www.math.ncsu.edu/~jdhauens}}

\author{Frank Sottile}
\address{Frank Sottile\\
         Department of Mathematics\\
         Texas A\&M University\\
         College Station\\
         Texas \ 77843\\
         USA}
\email{sottile@math.tamu.edu}
\urladdr{\url{http://www.math.tamu.edu/~sottile}}
\thanks{Research of both authors supported by NSF grant DMS-0915211 and Institut Mittag-Leffler
          (Djursholm, Sweden).
Research of Hauenstein also supported by NSF grants DMS-1114336 and DMS-1262428.}

\begin{document}

\begin{abstract}
 We present two algorithms that compute the Newton polytope of a polynomial defining a
 hypersurface $\sH$ in $\bC^n$ using numerical computation.
 The first algorithm assumes that we may only compute values of $f$---this may
 occur if $f$ is given as a straight-line program, as a determinant, or as an oracle.
 The second algorithm assumes that $\sH$ is represented numerically via a witness set.
 That is, it computes the Newton polytope of $\sH$ using only the ability to compute numerical
 representatives of its intersections with lines.
 Such witness set representations are readily obtained when $\sH$ is the image of a map or
 is a discriminant.
 We use the second algorithm to compute a face of the Newton polytope of the L\"uroth
 invariant, as well as its restriction to that face.
\end{abstract}

\maketitle

%
\section*{Introduction}\label{Sec:intro}

While a hypersurface $\sH$ in $\bC^n$ is always defined by the vanishing of a single
polynomial $f$, we may not always have access to the monomial representation
of $f$.
This occurs, for example, when $\sH$ is the image of a map or if $f$ is represented
as a straight-line program, and
it is a well-understood and challenging problem to determine the polynomial $f$ when $\sH$
is represented in this way.
Elimination theory gives a symbolic method based on Gr\"obner bases that can determine $f$
from a representation of $\sH$ as the image of a map or as a discriminant~\cite{CLO}.
Such computations require that the map be represented symbolically, and they may be
infeasible for moderately-sized input.

The set of monomials in $f$, or more simply the convex hull of their exponent
vectors (the Newton polytope of $f$), is an important combinatorial invariant of
the hypersurface.  The Newton polytope encodes asymptotic information
about $\sH$ and determining it from $\sH$ is a step towards determining the polynomial $f$.
For example, numerical linear algebra~\cite{CGKW,EK05} may be used to find $f$ given
its Newton polytope.  Similarly, the Newton polytope of an image of a map may be
computed from Newton polytopes of the polynomials defining
the map~\cite{IKL10,ES10,EK08,ST08,STY}, and computed using tropical geometric
algorithms~\cite{SY}.

We propose numerical methods to compute the Newton polytope of $f$
in two cases when $f$ is not known explicitly.
We first show how to compute the Newton polytope when we are able to evaluate $f$.
This occurs, for example, if $f$ is represented as a straight-line program or as a
determinant (neither of which we want to expand as a sum of monomials), or perhaps as a
compiled program.
For the other case, we suppose that $f$
defines a hypersurface $\sH$ that is represented numerically as a witness set.
Our basic idea is similar to ideas from tropical geometry.
The tropical variety of a hypersurface $\sH$ in $(\bC^\times)^n$ is the normal fan to the
Newton polytope of a defining polynomial $f$, augmented with the edge lengths.
The underlying fan coincides with the logarithmic limit set~\cite{Berg,BiGr} of $\sH$, which
records the asymptotic behavior of $\sH$ in $(\bC^\times)^n$.
We use numerical nonlinear algebra to study the asymptotic behavior of  $\sH$
in $(\bC^\times)^n$ and use this to recover the Newton polytope of a defining equation of $\sH$.
As both our algorithms are easily parallelizable, this numerical approach to Newton
polytopes should allow the computation of significantly larger examples than are possible with
purely symbolic methods.

This paper is organized as follows.
In Section~\ref{Sec:Supporting}, we explain symbolic and geometric-combinatorial
preliminaries, including representations of polytopes, Newton polytopes, and straight-line programs.
In Section~\ref{Sec:WitnessSets}, we discuss the essentials of numerical nonlinear algebra
(also called numerical algebraic geometry~\cite{SW05}), in particular explaining the
fundamental data structure of witness sets.
Our main results are in the next two sections.
In Section~\ref{Sec:NPeval} we explain (in Theorem~\ref{Thm:EvalMax} and
Remark~\ref{Rem:evaluate}) how to compute the Newton polytope of $f$, given
only that we may numerically evaluate $f$, and in Section~\ref{Sec:NPwitness}, we explain
(in Theorems~\ref{thm:bounded} and~\ref{thm:unbounded}, and Remark~\ref{Rem:witness}) how to use
witness sets to compute the Newton polytope of $f$.
Illustrative examples are presented in these sections.
In Section~\ref{Sec:Examples}, we combine our approach with other techniques
in numerical nonlinear algebra to explicitly compute the hypersurface of even L\"uroth quartics.

%
\section{Polynomials and Polytopes}\label{Sec:Supporting}

We explain necessary background from geometric combinatorics and algebra.

%
\subsection{Polytopes}
A polytope $P$ is the convex hull of finitely many points $\sA\subset\bR^n$,
 \begin{equation}\label{Eq:convex}
   P\ =\ \conv(\sA)\ := \Bigl\{ \sum_{\alpha\in\sA} \lambda_\alpha \alpha \::\:
      \lambda_\alpha\geq 0\,,\ \sum_\alpha \lambda_\alpha=1\Bigr\}\,.
 \end{equation}
Dually, a polytope is the intersection of finitely many halfspaces in $\bR^n$,
 \begin{equation}\label{Eq:Halfspace}
   P\ =\ \{ x\in\bR^n\::\: w_i\cdot x\leq b_i\quad\mbox{for }i=1,\dotsc,N\}\,,
 \end{equation}
where $w_1,\dotsc,w_N\in\bR^n$ and $b_1,\dotsc,b_N\in \bR$.
These are two of the most common representations of a polytope.
The first~\eqref{Eq:convex} is the \demph{convex hull} representation and the
second~\eqref{Eq:Halfspace} is the \demph{halfspace} representation.
The classical algorithm of Fourier-Motzkin elimination converts between these
two representations.

The \demph{affine hull} of a polytope $P$ is the smallest affine-linear space containing $P$.
The boundary of $P$ (in its affine hull) is a union of polytopes of
smaller dimension than $P$, called \demph{faces} of $P$.
A \demph{facet} of $P$ is a maximal proper face, while a \demph{vertex} is a minimal face
of $P$ (which is necessarily a point).
An \demph{edge} is a 1-dimensional face.

In addition to the two representations given above, polytopes also have a \demph{tropical
  representation}, which consists of the edge lengths, together with the normal fan to the
edges.
(This normal fan encodes the edge-face incidences.)
Jensen and Yu~\cite{JY} gave an algorithm for converting a tropical representation into a
convex hull representation.

Every linear function $x\mapsto w\cdot x$ on $\bR^n$ (here, $w\in\bR^n$) achieves a
maximum value on a polytope $P$.
The subset  \demph{$P_w$} of $P$ where this maximum value is achieved is a face of
$P$, called the \demph{face exposed by $w$}.
Let $h_P(w)$ be this maximum value of $w\cdot x$ on $P$.
The function $w\mapsto h_P(w)$ is called the \demph{support function} of $P$.
The support function encodes the halfspace representation as
\[
   P\ =\ \{ x\in\bR^n\::\: w\cdot x\leq h_P(w)\quad\mbox{for }w\in\bR^n\}\,.
\]

The \demph{oracle} representation is a fourth natural representation of a polytope $P$.
There are two versions.
For the first, given $w\in\bR^n$, if the face $P_w$ exposed by $w$ is a vertex, then it
returns that vertex, and if $P_w$ is not a vertex, it either returns a vertex on $P_w$ or
detects that $P_w$ is not a vertex.
Alternatively, it returns the value $h_P(w)$ of the support function at $w$.
The classical beneath-beyond algorithm~\cite[\S 5.2]{Gru03} uses an oracle representation
of a polytope to simultaneously construct its convex-hull and halfspace representations.
It iteratively  builds a description of the polytope, including the faces and
facet-supporting hyperplanes, adding one vertex at a time.
The software package iB4e~\cite{iB4e} implements this algorithm.
Another algorithm converting the oracle representation to the convex hull and halfspace
representation is ``gift-wrapping''~\cite{CK70}.

Our numerical algorithms return oracle representations.

%
\subsection{Polynomials and their Newton polytopes}
Let $\defcolor{\bN}=\{0,1,\dotsc\}$ be the nonnegative integers and
write $\defcolor{\bC^\times}$ for the nonzero complex numbers.
Of the many ways to represent a polynomial $f\in\bC[x_1,\dotsc,x_n]$, perhaps the most familiar
is in terms of monomials.
For $\alpha\in\bN^n$, we have the monomial
\[
   \defcolor{x^\alpha}\ :=\ x_1^{\alpha_1} x_2^{\alpha_2}\dotsb x_n^{\alpha_n}\,,
\]
which has \demph{degree} $\defcolor{|\alpha|}:=\alpha_1+\dotsb+\alpha_n$.
A polynomial $f$ is a linear combination of monomials
 \begin{equation}\label{Eq:linearCombin}
   f\ =\ \sum_{\alpha\in\bN^n} c_\alpha x^{\alpha}\qquad c_\alpha\in\bC\,,
 \end{equation}
where only finitely many coefficients $c_\alpha$ are nonzero.
The set $\{\alpha\in\bN^n \::\: c_\alpha\neq 0\}$ is the \demph{support} of $f$,
which we will write as \defcolor{$\sA(f)$}, or simply \defcolor{$\sA$} when $f$
is understood.

A coarser invariant of the polynomial $f$ is its \demph{Newton polytope}, $\sN(f)$.
This is the convex hull of its support
\[
    \defcolor{\sN(f)}\ :=\ \conv( \sA(f))\,.
\]
For $w\in\bR^n$, the \demph{restriction $f_w$} of $f$ to the face $\sN(f)_w$ of $\sN(f)$
exposed by $w$ is
 \begin{equation}\label{Eq:f_w}
   f_w\ :=\ \sum_{\alpha\in\sA\cap\sN(f)_w} c_\alpha x^{\alpha}\,,
 \end{equation}
the sum over all terms $c_\alpha x^\alpha$ of $f$ where $w\cdot \alpha$ is maximal
(and thus equal to $h_{\sN(f)}(w)$.)

A \demph{hypersurface} $\sH\subset\bC^n$ is defined by the vanishing of a single
polynomial, $\sH=\Var(f)$.
This polynomial $f$ is well-defined up to multiplication by non-zero scalars if we require it
to be of minimal degree among all polynomials vanishing on $\sH$.
We define the \demph{Newton polytope, $\sN(\sH)$}, of $\sH$ to be the Newton polytope of
any minimal degree polynomial $f\in\bC[x_1,\dotsc,x_n]$ defining $\sH$.\smallskip

Polynomials are not always given as a linear combination of
monomials~\eqref{Eq:linearCombin}.
For example, a polynomial may be given as a determinant whose entries are themselves polynomials.
It may be prohibitive to expand this into a sum of monomials, but it is computationally
efficient to evaluate the determinant.
For another example, a polynomial may be given as an oracle or as a compiled program.

An efficient encoding of a polynomial is as a \demph{straight line program}.
For a polynomial $f\colon\bC^n\to\bC$, this is a list
\[
   (f_{-n},\dotsc,f_{-1},f_0,f_1,\dotsc,f_{l})
\]
of polynomials where
$f=f_l$ and we have the initial values $f_{-i}=x_i$ for
$i=1,\dotsc,n$, and for every $k\geq 0$, $f_k$ is one of
\[
    f_i+f_j\,,\ f_i\cdot f_j\,,\ \mbox{or}\ c\,,
\]
where $i,j<k$ and $c\in\bQ[\sqrt{-1}]$ is a Gaussian rational number.
(Gaussian rational numbers are used for they are representable on a computer.)\smallskip

Our goal is twofold, we present an algorithm to compute the Newton polytope
of a polynomial $f$ that we can only evaluate numerically, and we present an algorithm to
recover the Newton polytope of a polynomial $f$ defining a hypersurface $\sH$ that is
represented numerically as a witness set (defined in \S~\ref{Sec:WitnessSets} below).

In the first case, we explain how to compute the support function $h_{\sN(f)}$ of the
Newton polytope of $f$, and to compute $\sN(f)_w$, when this is a vertex.
This becomes an algorithm, at least for general $w$, when we have additional information
about $f$, such as a finite superset $\sB\subset\bZ^n$ of its support and bounds on the
magnitudes of its coefficients.
This is discussed in Remark~\ref{Rem:evaluate}.

In the second case, we show how to compute $\sN(f)_w$, when this is a vertex.
This is discussed in Remark~\ref{Rem:witness}.

%
\section{Numerical nonlinear algebra and witness sets}\label{Sec:WitnessSets}

Numerical nonlinear algebra (also called numerical algebraic geometry~\cite{SW05}) provides
methods based on numerical continuation for studying algebraic varieties on a computer.
The fundamental data structure in this field is a
witness set, which is a geometric representation based on linear sections and generic points.

Given a polynomial system $F\colon \bC^m\rightarrow\bC^n$, consider an irreducible
component $V\subset{\defcolor{\Var(F)}}:=F^{-1}(0)$ of its zero set of dimension $k$ and degree $d$.
Let $\sL\colon\bC^m\rightarrow\bC^k$ be a system of general affine-linear polynomials
so that $\Var(\sL)$ is a general codimension $k$ affine subspace of $\bC^m$.
Then $\defcolor{W} := V\cap\Var(\sL)$ will consist of $d$ distinct points, and we call
the triple $(F,\sL,W)$ (or simply $W$ for short) a \demph{witness set} for $V$.
The set $W$ represents a general linear section of $V$.
Numerical continuation may be used to follow the points of $W$ as $\sL$ (and hence
$\Var(\sL)$) varies continuously.
This allows us to sample points from $V$.

Ideally, $V$ is a \demph{generically reduced} component of the scheme $\Var(F)$ in that the Jacobian
of $F$ at a general point $w\in W\subset V$ of $V$ has a $k$-dimensional null space.
Otherwise the scheme $\Var(F)$ is not reduced along $V$.
When $V$ is a generically reduced component of $\Var(F)$, the points of $W$ are nonsingular
zeroes of the polynomial system
$\left[\begin{smallmatrix} F\\ \sL\end{smallmatrix}\right]$.
When $\Var(F)$ is not reduced along $V$, the points of $W$ are singular zeroes of this
system, and it is numerically challenging to compute such singular points.

The method of deflation, building from~\cite{OWM83}, can compute $W$ when $\Var(F)$ is not reduced
along $V$.  In particular, the strong deflation method of~\cite{Isosingular} yields
a system $F'\colon\bC^m\rightarrow\bC^{n'}$ where $n'\geq n$ such
that $V$ is a generically reduced component of the scheme $\Var(F')$.
Replacing $F$ with $F'$, we will assume that $V$ is a generically reduced component of $\Var(F)$.

The notion of a witness set for the image of an irreducible variety under a linear map
was developed in~\cite{WitnessProj}.
Suppose that we have a polynomial system $F\colon\bC^m\rightarrow\bC^n$, a generically
reduced component $V$ of $\Var(F)$ of dimension $k$ and degree $d$, and a linear map
$\omega\colon\bC^m\rightarrow\bC^p$ defined by $\omega(x) = Ax$ for $A\in\bC^{p\times m}$.
Suppose that the algebraic set $U = \overline{\omega(V)}\subset\bC^p$ has dimension $k'$
and degree $d'$.
A witness set for the projection $U$ requires an affine-linear map $\sL$ adapted to the projection
$\omega$.
Let $\defcolor{B}$ be a matrix $\left[\begin{smallmatrix}B_1\\B_2\end{smallmatrix}\right]$
where the rows of the matrix $B_1\in\bC^{k'\times m}$ are general vectors in the row space of $A$
and the rows of $B_2\in\bC^{(k-k')\times m}$ are general vectors in $\bC^m$.
Define $\sL\colon\bC^m\rightarrow\bC^k$ by $\sL(x) = Bx - 1$ and set
$W:= V\cap\Var(\sL)$.
Then the quadruple $(F,\omega,\sL,W)$ is a \demph{witness set for the projection $U$}.
By our choice of $B$, the number of points in $\omega(W)$ is the degree $d'$ of $U$
and for any fixed $u\in\omega(W)$, the number of points in $W\cap\omega^{-1}(u)$
is the degree of the general fiber of $\omega$ restricted to $V$.
Note that $k - k'$ is the dimension of the general fiber.

\begin{example}\label{E:quadratic} Consider the discriminant hypersurface $\sH\subset\bC^3$
for univariate quadratic polynomials, that is, $\sH := \Var(f)$ where $f(a,b,c) = b^2 - 4ac$.
The triple $(f,\sL,W)$ where
\[
   \sL(a,b,c)\ :=\  \left[\begin{array}{c} 2a - 2b + 3c - 1 \\ 3a + b - 5c -
       1 \end{array}\right]
\]
and $W = \sH\cap\Var(\sL)$, which consists of the two points, $(a,b,c)$,
\[
   \{(0.3816, -0.1071, 0.00752)\,,\
     (1.2243,  2.1801, 0.97058)\}\,,
\]
is a witness set for $\sH$.

This discriminant also has the form $\sH = \overline{\omega(V)}$ where $\omega$ is the linear
projection mapping  $(a,b,c,x)$ to $(a,b,c)$ and $V = \Var(F)$ where
\[
  F(a,b,c,x)\ =\ \left[\begin{array}{c} ax^2 + bx + c \\ 2ax + b \end{array}\right]\,.
\]
This variety $V$ has dimension $2$ and degree $3$, and $\omega$ is defined by the matrix
\[
  A = \left[\begin{array}{cccc} 1 & 0 & 0 & 0 \\ 0 & 1 & 0 & 0 \\ 0 & 0 & 1 & 0 \end{array}\right].
\]
The quadruple $(F,\omega,\sL',W')$ where
$\sL'(a,b,c,x) = \sL(a,b,c)$ and $W' = V\cap\Var(\sL')$, which also consists of two
points, $(a,b,c,x)$,
\[
   \{(0.3816, -0.1071, 0.00752,  0.1403882)\,,\
     (1.2243,  2.1801, 0.97058, -0.8903882)\}\,,
\]
is also a witness set for $\sH$.
In particular, $\omega(W') = W$ and we see that $\omega$ restricted to $V$ is generically one-to-one.
\hfill\qed
\end{example}

%
\section{Newton polytopes via evaluation}\label{Sec:NPeval}

We address the problem of computing the Newton polytope of a polynomial $f:\bC^n\rightarrow\bC$
when we have a method to evaluate $f$.
This is improved when we have some additional information about the polynomial $f$.

For $t$ a positive real number and $w\in\bR^n$, set
$\defcolor{t^w}:=(t^{w_1},t^{w_2},\dotsc,t^{w_n})$.
Consider the monomial expansion of the polynomial $f$,
\[
   f\ =\ \sum_{\alpha\in\sA} c_\alpha x^\alpha\qquad \mbox{where}\quad c_\alpha\in\bC^\times\,.
\]
For  $x\in\bC^n$, we define
\[
    t^w\defcolor{.} x\ :=\ (t^{w_1}x_1\,,\, t^{w_2}x_2\,,\,\dotsc\,,\,t^{w_n}x_n)\,,
\]
 the coordinatewise product, and consider the evaluation,
 \begin{equation}\label{Eq:evaluate_translate}
   f(t^w. x)\ =\
     \sum_{\alpha\in\sA} c_\alpha t^{w\cdot\alpha} x^\alpha\,.
  \end{equation}

Let $\defcolor{\sF}:=\sN(f)_w$ be the face of $\sN(f)$ that is exposed by $w$.
Then if $\alpha\in\sF$, we have $w\cdot\alpha=h_{\sN(f)}(w)$.
There is a positive real number $\defcolor{d_w}$ such that if $\alpha\in\sA\smallsetminus\sF$,
then $w\cdot\alpha\leq h_{\sN(f)}(w)-d_w$.
Thus~\eqref{Eq:evaluate_translate} becomes
 \begin{eqnarray*}
   f(t^w. x) &=&  \sum_{\alpha\in\sA\cap\sF} c_\alpha t^{w\cdot\alpha} x^\alpha
    \ +\  \sum_{\alpha\in\sA\smallsetminus\sF} c_\alpha t^{w\cdot\alpha} x^\alpha\\
   &=& t^{h_{\sN(f)}(w)} \Bigl( f_w(x)\ +\
     \sum_{\alpha\in\sA\smallsetminus\sF}
  c_\alpha t^{w\cdot\alpha-h_{\sN(f)}(w)} x^\alpha\Bigr)\,,
 \end{eqnarray*}
where $f_w$ is the restriction of $f$ to the face $\sF$.
Observe that no exponent of $t$ which occurs in the sum exceeds $-d_w$.
This gives an asymptotic expression for $t\gg 0$,
 \begin{equation}\label{Eq:asymptotic_expansion}
   \log| f(t^w.x)|\ =\
   h_{\sN(f)}(w)\log(t)\ +\ \log|f_w(x)|\ +\ O(t^{-d_w})\,,
 \end{equation}
from which we deduce the following limit.

\begin{lemma}\label{L:Evaluation_Limit}
  If $f_w(x)\neq 0$, then
\[
   h_{\sN(f)}(w)\ =\
       \lim_{t\to\infty} \frac{\log|f(t^w. x)|}{\log(t)}\,.
\]
\end{lemma}

Thus we may approximate the support function of $\sN(f)$ by evaluating $f$ numerically.

\begin{remark}\label{Re:evaluation}
 To turn Lemma~\ref{L:Evaluation_Limit} into an algorithm for computing $h_{\sN(f)}$, we need
 more information about $f$, so that we may estimate the rate of convergence.
 For example, if we have a bound, in the form of a finite superset $\sB\subset\bN^n$ of
 $\sA$, then $\{w\cdot\alpha\mid\alpha\in\sB\}$ is a discrete set which contains
 the value of $h_{\sN(f)}(w)$, and therefore the limit in Lemma~\ref{L:Evaluation_Limit}.~\hfill\qed
\end{remark}

 When $w$ is generic in that $\alpha\mapsto w\cdot\alpha$ is injective on $\sA$, then
 the face $\sN(f)_w$ of $\sN(f)$ exposed by $w$ is a vertex so that $f_w(x)\neq 0$ for any
 $x\in(\bC^\times)^n$.
 We may dispense with the limit given an {\it a priori} estimate on the
 magnitude of the coefficients of $f$.

\begin{theorem}\label{Thm:EvalMax}
 Let $f\colon \bC^n\rightarrow\bC$ be a polynomial with monomial expansion,
  \[
   f(x)\ =\ \sum_{\alpha\in\sA} c_\alpha x^\alpha\,,
    \qquad c_\alpha\in\bC^\times\,.
 \]
 Suppose that $\delta,\lambda\geq 1$ and $\sB\subset\bN^n$ are such that
 \begin{enumerate}
  \item\label{item:delta} $\log|c_\alpha| \leq \delta$ for all $\alpha\in\sA$,
  \item\label{item:lambda} $\log|c_\alpha| \leq \lambda +\log|c_\beta|$ for all
             $\alpha,\beta\in\sA$, and
  \item\label{item:convexA} $\sA\subset\sB$ with $|\sB| < \infty$.
 \end{enumerate}
 Let $w\in\bR^N$ be general in that
 \[
   \defcolor{d_w}\ :=\ \min_{\alpha\neq\beta\in\sB} |w\cdot \alpha - w\cdot \beta|\ >\ 0\,.
 \]
 Then the face of $\sN(f)$ exposed by $w$ is a vertex which equals the unique $\beta\in\sB$
 such that
 \[
  \left|w\cdot\beta - \frac{\log |f(t^w)|}{\log(t)}\right|\ <\ \frac{d_w}{2}
 \]
 where $t>0$ is any number with $\log(t)$ exceeding
 $\max\{2\lambda, 2(\delta+e^{-1}), \lambda + \log|\sB|+1\}/d_w$.
 Similarly, the face exposed by $-w$ is a vertex which equals the unique $\beta\in\sB$
 such that
 \[
  \left|-w\cdot\beta - \frac{\log |f(t^{-w})|}{\log(t)}\right|\ <\ \frac{d_w}{2}
 \]
 where $t>0$ is any number with $\log(t)$ exceeding
 $\max\{2\lambda,2(\delta+e^{-1}), \lambda + \log|\sB|+1\}/d_w$.
\end{theorem}

\begin{remark}\label{Rem:evaluate}
 Suppose that we know or may estimate the quantities $\sB$, $\delta$, and $\lambda$
 of Theorem~\ref{Thm:EvalMax}.
 Then, for general $w\in\bR^n$ we may compute $d_w$, and therefore evaluating
 $\log|f(t^w)|/\log(t)$ for $t^{d_w}>\max\{e^{2\lambda},e^{2+2\delta},|\sB|e^{\lambda+1}\}$ will yield
 $w\cdot\beta$ and hence $\beta$.

 Even without this knowledge, we may still compute the support function
 $h_{\sN(f)}(w)$ for $w\in\bQ^n$ as follows.
 For $0\neq w\in\bQ^n$ the map $\bZ^n\to\bQ$ given by $\beta\mapsto w\cdot\beta$ has image a
 free group $\bZ d_w$ for some $d_w>0$.
 For $x\in\bC^n$ with $f_w(x)\neq 0$ and $t:=e^\tau$ with $\tau>0$, we have
\[
  \left| \frac{\log|f( e^{\tau w}.x)|}{\tau}\ -\ h_{\sN(f)}(w)\right|
    \ \ \approx\ \
  \frac{\log|f_w(x)|}{\tau}\ +\ O(e^{-d_w\tau})\,.
\]

 Since $h_{\sN(f)}(w)\in \bZ d_w$, we may do the following.
 Pick a general $x\in\bC^n$ (so that $f_w(x)\neq 0$), and compute the quantity
 \begin{equation}\label{Eq:quantity}
   \frac{\log|f( e^{\tau w}.x)|}{\tau}
 \end{equation}
 for $\tau$ in some increasing sequence of positive numbers.
 We monitor~\eqref{Eq:quantity} for $\frac{1}{\tau}$-convergence to some
 $\kappa d_w\in\bZ d_w$.
 Then $h_{\sN(f)}(w)=\kappa d_w$.

 Every such computation gives a halfspace
\[
   \{x\in\bR^n \,\mid\, w\cdot x\leq h_{\sN(f)}(w)\}
\]
 containing $\sN(f)$.
 Since $\sN(f)$ lies in the positive orthant, we may repeat this one or more times to
 obtain a bounded polytope $P$ containing $\sN(f)$.
 Having done so, set $\sB:=P\cap\bN^n$.

 Suppose that $w\in\bR^n$ is general in that the values of $w\cdot \alpha$ for $\alpha\in\sB$
 are distinct.
 This implies that $w$ exposes a vertex $\beta$ of $\sN(f)$.
 Then a  similar (but simpler as  $f_w(t^w)=c_\beta t^{h_{\sN(f)}(w)}$) scheme as described
 above will result in the computation of the support function $h_{\sN(f)}(w)$ and the
 vertex $\beta$.
 \hfill\qed
\end{remark}

\begin{proof}[Proof of Theorem~$\ref{Thm:EvalMax}$]
 By the choice of $w$, the face of $\sN(f)$ it exposes is a vertex, say $\beta\in\sA$, and
 we have $w\cdot\beta=h_{\sN(f)}(w)$.
 We may write
 \[
    f(t^w)\ =\ c_\beta t^{w\cdot\beta}\ +\ \bigl(f(t^w)-c_\beta t^{w\cdot\beta}\bigr)\ =\
    c_\beta t^{w\cdot\beta}\left(1\ +\ \frac{f(t^w)-c_\beta t^{w\cdot\beta}}{c_\beta t^{w\cdot\beta}}\right)\,.
 \]
Taking absolute value and logarithms, and using that $\log|c_\beta|<\delta$ and
$w\cdot\beta=h(w)$,
 \begin{eqnarray}
   \log |f(t^w)| & = & \log \left|c_{\beta} t^{w\cdot\beta}\right| +
     \log\left|1 + \frac{f(t^w) - c_{\beta} t^{w\cdot\beta}}{c_{\beta}
     t^{w\cdot\beta}}\right|  \label{Eq:absoluteValues}
    \\
      &\leq& \delta + w\cdot\beta\,\log(t)
    + \log \left|1+ \frac{f(t^w) - c_{\beta} t^{w\cdot\beta}}{c_{\beta}
     t^{w\cdot\beta}}\right|\,.  \nonumber
 \end{eqnarray}
 Let us estimate the last term.
 As $\sA\subset\sB$, we have
\[
 \left|\frac{f(t^w) - c_{\beta} t^{w\cdot\beta}}{c_{\beta} t^{w\cdot\beta}}\right|
  \ =\
 \left|\sum_{\substack{\alpha\in\sA \\ \alpha \neq \beta}}
    \frac{c_\alpha}{c_{\beta}} t^{w\cdot\alpha-w\cdot\beta}\right|
  \ \leq\
 \sum_{\substack{\alpha\in\sA \\ \alpha\neq\beta}} e^{\lambda}t^{-d_w}
  \ \leq\
 |\sB|\,e^{\lambda - \log(t) d_w} \,.
\]
 Since $\log(t) > (\lambda + \log |\sB|+1)/d_w$, we have
  $ |\sB|e^{\lambda - \log(t)d_w} < e^{-1}$.
 Since $\log|1+x| \leq |x|$, we have
\[
  \log\left|1+ \frac{f(t^w) - c_{\beta} t^{w\cdot\beta}}{c_{\beta} t^{w\cdot\beta}}\right|
   \ \leq\
   \left|\frac{f(t^w) - c_{\beta} t^{w\cdot\beta}}{c_{\beta} t^{w\cdot\beta}}\right|
   \ \leq\ |\sB| \,e^{\lambda - d_w \log(t)} \ <\  e^{-1}\,.
\]
 Finally, as we have $\log(t) > 2(\delta + e^{-1})/d_w$, we obtain
 \begin{equation}\label{Eq:half}
   \frac{\log|f(t^w)|}{\log(t)} \leq w\cdot\beta + (\delta + e^{-1})\frac{1}{\log(t)}
   \ < \ w\cdot\beta + \frac{d_w}{2}.
 \end{equation}
 For the other inequality, using~\eqref{Eq:absoluteValues} and Condition (2) of the
 theorem,
 \begin{equation}\label{Eq:lower_est}
   \log|f(t^w)| \ \geq\ \delta-\lambda+\log(t)\, w\cdot\beta
    + \log \left|1+ \frac{f(t^w) - c_{\beta} t^{w\cdot\beta}}{c_{\beta}
     t^{w\cdot\beta}}\right|\,.
 \end{equation}
 Since
\[
   \left|\frac{f(t^w) - c_{\beta} t^{w\cdot\beta}}{c_{\beta}
     t^{w\cdot\beta}}\right| \ <\ e^{-1}\,,
\]
 the logarithm on the right of~\eqref{Eq:lower_est} exceeds $-1$.
 As $\delta-\lambda\geq 1-\lambda$, we have
\[
  \frac{\log|f(t^w)|}{\log(t)}\ >\
   -\frac{\lambda}{\log(t)} + w\cdot\beta\ \geq\
   w\cdot\beta-\frac{d_w}{2}\,,
\]
 since $d_w\log(t)\geq 2\lambda$.
 Combining this with~\eqref{Eq:half} proves the first statement about $f(t^w)$.
 The statement about $f(t^{-w})$ has the same proof, replacing $w$ with $-w$.
\end{proof}

\begin{example}\label{Ex:DiscEval}
Reconsider the polynomial $f(a,b,c) = b^2 - 4ac$
from Ex.~\ref{E:quadratic} with the vector $w = (-1.2,0.4,3.7)$.
Suppose that we take $\lambda = \delta = 2$ and $\sB = \{a^2,ab,ac,b^2,bc,c^2\}$ which are
the columns of the matrix
\[
  \sB\ =\
  \left(\begin{matrix}2&1&1&0&0&0\\
                      0&1&0&2&1&0\\
                      0&0&1&0&1&2\end{matrix}\right)\,.
\]
Then the dot products are $w\cdot \sB=(-2.4,-0.8,2.5,0.8,4.1,7.4)$,
so that $d_w = 1.6$.
Since we need $\log(t) > 3.75$, we can take $t=45$, and so
$t^w=\left(45^{-1.2}, 45^{0.4},  45^{3.7}\right)$.
We compute
\[
  \frac{\log |f(t^w)|}{\log(t)}\ =\ 2.864
   \qquad \hbox{and} \qquad
  -\frac{\log |f(t^{-w})|}{\log(t)}\ =\ 0.8016 \,.
\]
Thus, the monomials $ac$ and $b^2$ are the vertices $\sN(f)_w$ and $\sN(f)_{-w}$, respectively.~\hfill\qed
\end{example}

%
\section{Newton polytopes via witness sets}\label{Sec:NPwitness}

Let $\sH\subset \bC^n$ be an irreducible hypersurface and suppose that we have a witness set
representation for $\sH$.
As discussed in Section~\ref{Sec:WitnessSets}, this means that we may compute the intersections
of $\sH\cap \ell$ where $\ell$ is a general line in $\bC^n$.
We explain how to use this information to compute an oracle representation of the Newton
polytope of $\sH$.

The hypersurface $\sH\subset\bC^n$ is defined by a single irreducible polynomial
 \begin{equation}\label{Eq:poly_f}
   f\ =\ \sum_{\alpha\in\sA} c_\alpha x^\alpha \qquad
   c_\alpha\in \bC^\times\,,
 \end{equation}
which is determined by $\sH$ up to multiplication by a scalar.

Let $a,b\in\bC^n$ be general points, and consider the parametrized line
\[
  \defcolor{\ell_{a,b}}\ =\ \defcolor{\ell(s)}\ :=\ \{sa-b\mid s\in\bC\}\,.
\]
Then the solutions to $f(\ell(s))=0$ parameterize the intersection of $\sH$ with the line
$\ell_{a,b}$, which is a witness set for $\sH$.

Let $w\in\bR^n$.
For $t$ a positive real number,
consider $f(t^w.\ell(s))$, which is
 \begin{equation}\label{Eq:swellt}
   \sum_{\alpha\in\sA} c_\alpha
    (sa_1-b_1)^{\alpha_1} (sa_2-b_2)^{\alpha_2} \dotsb (sa_n-b_n)^{\alpha_n}\,.
    t^{w\cdot\alpha}
 \end{equation}
Write $\defcolor{(sa-b)^\alpha}$ for the product of terms $(sa_i-b_i)^{\alpha_i}$ appearing in the
sum.

Let $\defcolor{\sF}:=\sN(\sH)_w$ be the face of the Newton polytope of $\sH$ exposed by $w$.
If $\alpha\in \sF$, then $w\cdot \alpha=h(w)$, where $h$ is the support function of
$\sN(\sH)$.
There is a positive number $d_w$ such that if $\alpha\in \sA\smallsetminus \sF$, then
$w\cdot \alpha\leq h(w)-d_w$.
We may rewrite~\eqref{Eq:swellt},
 \[
   f(t^w. \ell(s))\ =\
     t^{h(w)}\sum_{\alpha\in \sA\cap \sF} c_\alpha (as-b)^\alpha
     \ +\ \sum_{\alpha\in\sA\smallsetminus \sF}c_\alpha (as-b)^\alpha t^{w\cdot\alpha}\ .
 \]
Multiplying by $t^{-h(w)}$ and rewriting using the definition~\eqref{Eq:f_w} of $f_w$ gives
 \begin{equation}\label{Eq:separated}
   t^{-h(w)}f(t^w. \ell(s))\ =\ f_w(\ell(s))
     \ +\ \sum_{\alpha\in\sA\smallsetminus\sF}c_\alpha (as-b)^\alpha t^{w\cdot\alpha-h(w)}\ .
 \end{equation}
Observe that the exponent of $t$ in each term of the sum over $\sA\smallsetminus\sF$ is at most
$-d_w$.

As $s\mapsto\ell(s)$ and $s\mapsto t^w.\ell(s)$ are general parametrized lines in $\bC^n$, the
zeroes (in $s$) of $f(t^w.\ell(s))$ and $f_w(\ell(s))$ parameterize witness sets for
$f$ and $f_w$, respectively.
The following summarizes this discussion.

\begin{lemma}\label{L:facial_limit}
  In the limit as $t\to\infty$, there are $\deg(f)-\deg(f_w)$  points of the witness set
  $f(t^w.\ell(s))=0$ which diverge to $\infty$ (in $s$) and the remaining points
  converge to the witness set $f_w(\ell(s))=0$.
\end{lemma}

When $\sN(\sH)_w$ is a vertex $\beta$, then
$f_w=c_\beta x^\beta$ (and $\deg(f_w)=|\beta|$), and
\[
   f_w(\ell(s))\ =\ c_\beta (sa_1-b_1)^{\beta_1}(sa_2-b_2)^{\beta_2}
                      \dotsb (sa_n-b_n)^{\beta_n}\ =\
     c_\beta(sa-b)^\beta\,.
\]
In particular, there will be $\beta_i$ points of $f(t^w.\ell(s))=0$ converging to $b_i/a_i$ as
$t\to \infty$, and so Lemma~\ref{L:facial_limit} gives a method to compute the vertices
$\beta$ of $\sN(\sH)$.
We give some definitions to make these notions more precise.

Let $a\in(\bC^\times)^n$ and $b\in\bC^n$ be general in that the univariate polynomial
$f(\ell_{a,b}(s))$ has $d=\deg(\sH)$ nondegenerate roots, and if $i\neq j$, then
$b_i/a_i\neq b_j/a_j$.
For any $w\in\bR^n$ with $\sN(\sH)_w=\{\beta\}$, consider the bivariate function
$\defcolor{g_{a,b,w}(s,t)}=g(s,t):=f(t^w.\ell_{a,b}(s))$.
Since $g(s,1)$ has $d$ simple zeroes, there are at most finitely many positive
numbers $t$ for which $g(s,t)$ does not have $d$ simple zeroes.
Therefore, there is a $t_0>0$ and $d$ disjoint analytic curves $s(t)\in\bC$ for $t>t_0$
which parameterize the zeroes of $g(s,t)$ for $t>t_0$
(that is, $g(s(t),t)\equiv 0$ for $t>t_0$).

By Lemma~\ref{L:facial_limit} and our choice of $a,b$, for each $i=1,\dotsc,n$, exactly
$\beta_i$ of these curves will converge to $b_i/a_i$ as $t\to\infty$, for each $i=1,\dotsc,n$,
while the remaining $d-|\beta|$ curves will diverge to infinity.
We give an estimate of the rates of these convergences/divergences.
\smallskip

Let $w\in\bR^n$ be general in that $\sN(\sH)_w$ is a vertex, $\beta$.
Let $d_w$ be as above, and set
\[
   \defcolor{C}\ :=\ \frac{\max\{|c_\alpha|\::\: \alpha\in\sA\}}{|c_\beta|}\,.
\]
Furthermore, set
$\defcolor{a_{\min}}:=\min\{1,|a_i|\::\: i=1,\dotsc,n\}$,
$\defcolor{a_{\max}}:=\max\{1,|a_i|\::\: i=1,\dotsc,n\}$, and the same,
$\defcolor{b_{\min}}$ and $\defcolor{b_{\max}}$, for $b$.
Finally, for each $i=1,\dotsc,n$, define
 \begin{eqnarray*}
  \gamma_i&:=& \min\left\{ a_{\min}\,,\,
           \frac{1}{2}\left|\frac{b_i}{a_i}-\frac{b_j}{a_j}\right| \::\: i\neq j\right\}
         \ ,\qquad\mbox{\rm and}\\
  \Gamma_i&:=& \max\left\{ \frac{2}{a_{\max}}\,,\,
           \left|\frac{b_i}{a_i}-\frac{b_j}{a_j}\right| \::\: i\neq j\right\}\ .
            \rule{0pt}{20pt}
 \end{eqnarray*}

We give two results about the rate of convergence/divergence of the analytic curves $s(t)$ of
zeroes of $g(s,t)$, and then discuss how these may be used to compute $\sN(\sH)$.

\begin{theorem}\label{thm:bounded}
 With the above definitions, suppose that $s\colon (t_0,\infty)\to \bC$ is a continuous
 function such that $g_{a,b,w}(s(t),t)\equiv 0$ for $t>t_0$ and that $s(t)$ converges to
 $b_i/a_i$ as $t\to\infty$.
 Let $t_1\geq t_0$ be a number such that if $t>t_1$ then
 \begin{equation}\label{Eq:first_estimate}
    \left| s(t)\ -\ \frac{b_i}{a_i}\right|\ \leq\ \gamma_i\,.
 \end{equation}
 Then, for all $t>t_1$,
 \begin{equation}\label{Eq:first_subexponential}
   \left| s(t)\ -\ \frac{b_i}{a_i}\right|^{\beta_i}\ \leq\
    t^{-d_w}\cdot C\cdot |\sA|\cdot
    \left(\frac{a_{\max}}{a_{\min}}\left(1 + \frac{\Gamma_i}{\gamma_i}\right)\right)^d\ .
 \end{equation}
\end{theorem}

\begin{theorem}\label{thm:unbounded}
 With the above definitions, suppose that $s\colon (t_0,\infty)\to \bC$ is a continuous
 function such that $g_{a,b,w}(s(t),t)=0$ for $t>t_0$ and that $s(t)$ diverges to $\infty$
 as $t\to\infty$.
 Let $t_1\geq t_0$ be a number such that if $t>t_1$ then
 \begin{equation}\label{Eq:second_estimate}
    | s(t)|\ >\ \frac{2 b_{\max}}{a_{\min}}\ \geq\ 2 \,.
 \end{equation}
 Then, for all $t>t_1$,
 \begin{equation}\label{Eq:second_subexponential}
   |s(t)|^{d-|\beta|}\ \geq\
    \frac{t^{d_w}}{C\cdot |\sA|}\cdot
    \left(\frac{a_{\min}}{2(a_{\max}+a_{\min})}\right)^d\ .
 \end{equation}
\end{theorem}

\begin{remark}\label{Rem:witness}
 Theorems~\ref{thm:bounded} and~\ref{thm:unbounded} lead to an algorithm to determine
 vertices of $\sN(\sH)$.
 First, choose $a,b\in\bC^n$ as above and compute $\gamma_i$, $b_{\max}$, and $a_{\min}$.
 For a general $w\in\bR^n$, follow points in the witness set $\sH\cap(t^w.\ell_{a,b}(s))$
 as $t$ increases until the inequalities~\eqref{Eq:first_estimate}
 and~\eqref{Eq:second_estimate} are satisfied by the different points of the witness sets,
 at some $t_1$.
 This will give likely values for the integer components of the vertex $\beta$ exposed by $w$.
 Next, continue following these points until the subexponential convergence
 in~\eqref{Eq:first_subexponential} and~\eqref{Eq:second_subexponential} is observed,
 which will confirm the value of $\beta$.

 If we do not observe clustering of points of the witness set at $s=b_i/a_i$ and
 $s=\infty$, then we discard $w$, as it is not sufficiently general.
 That is, either it exposes a positive dimensional face of $\sN(\sH)$ or else it is very
 close to doing so in that $d_w$ is too small.
   \hfill\qed
\end{remark}

\begin{proof}[Proof of Theorem~$\ref{thm:bounded}$]
 Fix $t>t_1$.
 Since $0=g_{a,b,w}(s(t),t)=f(t^w.\ell_{a,b}(s(t)))$ and
 $f_w(x)=c_\beta x^\beta$,~\eqref{Eq:separated} gives
 \begin{eqnarray}
   | (s(t) a  - b)^{\beta}| &\leq &  \nonumber
    \sum_{\alpha\in\sA\smallsetminus\{\beta\}} t^{w\cdot\alpha-w\cdot\beta}\cdot
      \frac{|c_\alpha|}{|c_\beta|}\cdot |(s(t) a - b)^\alpha|\\ \label{Eq:est_1}
   &\leq& t^{-d_w}\cdot C\cdot  \sum_{\alpha\in\sA\smallsetminus\{\beta\}} |(s(t) a -b)^\alpha|\,.
 \end{eqnarray}
  For any $i$ and $j$ we have
 \[
    |s(t) a_j - b_j|\ =\  |a_j|\cdot   |s(t) - \tfrac{b_j}{a_j}|\ \leq\ a_{\max}
      \bigl| s(t)-\tfrac{b_i}{a_i}\ +\ \tfrac{b_i}{a_i}-\tfrac{b_j}{a_j}\bigr|
      \ \leq\ a_{\max}(\gamma_i+\Gamma_i)\,.
 \]
 Since $2\leq a_{\max}\Gamma_i$ and if $\alpha\in\sA$, then $|\alpha|\leq d$, we have
 \begin{equation}\label{Eq:est_3}
    |(s(t) a -b)^\alpha|\ \leq\ \bigl(a_{\max}(\gamma_i+\Gamma_i)\bigr)^d\,.
 \end{equation}
 With~\eqref{Eq:est_1}, this becomes
 \begin{equation}\label{Eq:new_est}
   | (s(t) a - b)^{\beta}| \ \leq\
    t^{-d_w} \cdot C\cdot |\sA| \cdot \bigr(a_{\max}(\gamma_i+\Gamma_i)\bigr)^d\,.
 \end{equation}

 If $j\neq i$, then
 \begin{eqnarray*}
   |s(t) a_j - b_j| \ =\ |a_j|\cdot \bigl|s(t)- \tfrac{b_j}{a_j}\bigr|
    &=& |a_{j}|\cdot \bigl|s(t) - \tfrac{b_i}{a_i} + \tfrac{b_i}{a_i} - \tfrac{b_j}{a_j}\bigr|\\
    &\geq& a_{\min}\cdot \Bigl| \bigl|\tfrac{b_i}{a_i} - \tfrac{b_j}{a_j}\bigr|\ -\
                 \bigl| s(t) - \tfrac{b_i}{a_i} \bigr| \Bigr|\\
    &\geq& a_{\min}\cdot (2\gamma_i\ -\ \gamma_i)\ =\ a_{\min} \gamma_i\,.
 \end{eqnarray*}
 Since $a_{\min}\gamma_i\leq 1$ and $|\beta|\leq d$, we have
 \begin{equation}\label{Eq:est_2}
   \prod_{j\neq i} \bigl|(s(t) a_j - b_j)^{\beta_j}\bigr|
    \ \geq\ (a_{\min} \gamma_i)^{d-\beta_i}\,.
 \end{equation}
 Observe that we have
 \[
     \left| s(t)\ -\ \frac{b_i}{a_i}\right|^{\beta_i}\ =\
     \frac{1}{|a_i|^{\beta_i}}\cdot|s(t) a_i - b_i|^{\beta_i}\ =\
     \frac{1}{|a_i|^{\beta_i}} \cdot
     \frac{|(s(t)a-b)^\beta|}{\prod_{j\neq i} \bigl|(s(t)a_j-b_j)^{\beta_j}\bigr|}\,.
 \]
 Combining this with~\eqref{Eq:new_est} and~\eqref{Eq:est_2} gives
 \[
   \left| s(t)\ -\ \frac{b_i}{a_i}\right|^{\beta_i}\ \leq\
    t^{-d_w}\cdot C\cdot|\sA|\cdot \bigl(a_{\max}(\gamma_i+\Gamma_i)\bigr)^d
    \cdot\frac{1}{a_{\min}^{\beta_i}}\cdot
    \frac{1}{(a_{\min} \gamma_i)^{d-\beta_i}}\ .
 \]

 Since $1\geq a_{\min}\geq\gamma_i$ and $d\geq\beta_i\geq 0$, we have
 \[
   \left| s(t)\ -\ \frac{b_i}{a_i}\right|^{\beta_i}\ \leq\
    t^{-d_w}\cdot C\cdot |\sA|\cdot
    \left(\frac{a_{\max}}{a_{\min}}\left(1 + \frac{\Gamma_i}{\gamma_i}\right)\right)^d\ ,
 \]
 which completes the proof.
\end{proof}

\begin{proof}[Proof of Theorem~$\ref{thm:unbounded}$]
 Fix $t>t_1$.
 Then $|s(t)|>2$.
 Since $g_{a,b,w}(s(t),t)=0$, we have
\[
   | c_\beta t^{w\cdot\beta} (s(t)a-b)^\beta|\ =\
   \Bigl| \sum_{\alpha\in\sA\smallsetminus\{\beta\}}
        c_\alpha t^{w\cdot\alpha}(s(t)a-b)^\alpha\Bigr|
  \ \leq\ \sum_{\alpha\in\sA\smallsetminus\{\beta\}}
         |c_\alpha t^{w\cdot\alpha}(s(t)a-b)^\alpha|\,.
\]
 Factoring out powers of $|s(t)|$, we obtain
 \[
     t^{w\cdot\beta} |s(t)|^{|\beta|}|c_\beta|\, |(a-b\,s(t)^{-1})^\beta| \ \leq\
    \sum_{\alpha\in\sA\smallsetminus\{\beta\}} t^{w\cdot\alpha} |s(t)|^{|\alpha|}
     |c_\alpha|\,|(a-b\,s(t)^{-1})^\alpha|\,.
\]
 Since $s(t)\to\infty$ as $t\to\infty$, we must have $d-|\beta|>0$.
 Dividing by most of the left hand side and by $|s(t)|^d$ and using the definition of $d_w$, we
 obtain
 \begin{eqnarray}
  |s(t)|^{|\beta|-d}& \leq&  \nonumber
    \sum_{\alpha\in\sA\smallsetminus\{\beta\}} t^{w\cdot\alpha-w\cdot\beta}
      |s(t)|^{|\alpha|-d}
      \frac{|c_\alpha|}{|c_\beta|}\,
      \frac{|(a-b\,s(t)^{-1})^\alpha|}{|(a-b\,s(t)^{-1})^\beta|}\\ \label{Eq:Need_this}
   &\leq& t^{-d_w}\cdot C\cdot
          \sum_{\alpha\in\sA\smallsetminus\{\beta\}}  |s(t)|^{|\alpha|-d}
      \frac{|(a-b\,s(t)^{-1})^\alpha|}{|(a-b\,s(t)^{-1})^\beta|}\ .
 \end{eqnarray}
  We estimate the terms in this last sum.
  As $|s(t)|\geq 2$, for any $i$ we have
\[
   |a_i-b_i\,s(t)^{-1}|\ \leq\ |a_i| + |b_i\, s(t)^{-1}|
   \ \leq\ a_{\max} + b_{\max}\,,
\]
 and so $|(a-b\,s(t)^{-1})^\alpha|\leq (a_{\max}+b_{\max})^{|\alpha|}$.
 Similarly, for any $i$ we have
 \[
   |a_i-b_i\,s(t)^{-1}|\ \geq\ |a_i|-|b_i\,s(t)^{-1}| \ \geq\
    a_{\min}\ -\ b_{\max}\cdot\frac{a_{\min}}{2 b_{\max}}\ =\ \frac{a_{\min}}{2}\ .
 \]
 Thus
 \[
    \frac{|(a-b\,s(t)^{-1})^\alpha|}{|(a-b\,s(t)^{-1})^\beta|}\ \leq\
    (a_{\max}+b_{\max})^{|\alpha|} \bigl(\frac{2}{a_{\min}}\bigr)^{|\beta|}
     \ <\ \Bigr(\frac{2(a_{\max}+b_{\max})}{a_{\min}}\Bigr)^d\,.
 \]
 Substituting this into~\eqref{Eq:Need_this} completes the proof of the theorem.
\end{proof}

\begin{example}
We demonstrate the convergence and divergence bounds by considering the polynomial $f(x,y) = x^2+3x+2y-5$
with the hypersurface $\sH:=\Var(f)$ it defines.
We have $\sA = \{1,x,y,x^2\}$ with $|\sA| = 4$ and will take $C = 5$, $a = (2+\sqrt{-1},3-2\sqrt{-1})$, $b = (-1-\sqrt{-1},2-3\sqrt{-1})$,
$a_{\min} = 1$, $a_{\max} = \sqrt{13}$, $b_{\min} = 1$, and $b_{\max} = \sqrt{13}$.
Additionally, $\gamma_i = \Gamma_i \approx 1.5342$ for $i = 1,2$.

First, consider the vector $w = (1,1)$ for which $\sN(\sH)_w = (2,0)$ and $d_w = 1$.
We have $g_{a,b}(s,t)=f(t\cdot(sa_1-b_1),t\cdot(sa_2-b_2))$, and $g_{a,b}(s,t)=0$
has two nonsingular solutions for all $t > 0$.  Since $\sN(\sH)_w = (2,0)$ both
solutions paths converge to $b_1/a_1$ as $t\rightarrow\infty$.
The following table compares the actual values for the two solution paths, $s_1(t)$ and
$s_2(t)$, with the upper bound (\ref{Eq:first_subexponential}) in Theorem~\ref{thm:bounded}.
In particular, this table shows $|s_i(t) - b_1/a_1|^2 \approx 2.2 t^{-1}$
whereas the upper bound is $1040 t^{-1}$.

\smallskip
\[
  \begin{tabular}{|c|c|c|c|}
    \hline
    $t$ & $|s_1(t) - b_1/a_1|^2$ & $|s_2(t) - b_1/a_1|^2$ & Upper bound (\ref{Eq:first_subexponential}) \\

    \hline
    $1e2$ & 0.26 & 0.19 & 10.4 \\
    \hline
    $1e4$ & 2.2e-4 & 2.2e-4 & 0.104 \\
    \hline
    $1e6$ & 2.2e-6 & 2.2e-6 & 1.04e-3 \\
    \hline
    $1e8$ & 2.2e-8 & 2.2e-8 & 1.04e-5 \\
    \hline
  \end{tabular}
\]
\smallskip

We now consider the vector $w = (-1,-1)$ for which $\sN(\sH)_w = (0,0)$ and $d_w = 2$.
With the same $a,b$ as above, $g_{a,b}(s,t)=f(t^{-1}\cdot(sa_1-b_1), t^{-1}\cdot(sa_2-b_2))$
and $g_{a,b}(s,t)=0$ has two nonsingular solutions for all $t > 0$.
Since $\sN(\sH)_w = (0,0)$, both solution paths diverge to $\infty$ as $t\to\infty$.
The following table compares the actual values for the two solution paths, $s_1(t)$ and $s_2(t)$,
and the lower bound (\ref{Eq:second_subexponential}) in Theorem~\ref{thm:unbounded}.
This table shows $|s_i(t)|^2 \approx t^2/8.71$ whereas the lower bound is $t^2/4160$.

\smallskip
\[
  \begin{tabular}{|c|c|c|c|}
    \hline
    $t$ & $|s_1(t)|^2$ & $|s_2(t)|^2$ & Lower bound (\ref{Eq:second_subexponential})  \\

    \hline
    $1e2$ & 1.17e3 & 1.13e3  & 2.40 \\
    \hline
    $1e4$ & 1.15e7 & 1.15e7 & 2.40e4 \\
    \hline
    $1e6$ & 1.15e11 & 1.15e11 & 2.40e8 \\
    \hline
    $1e8$ & 1.15e15 & 1.15e15 & 2.40e12 \\
    \hline
  \end{tabular}
\]
\hfill\qed
\end{example}

%
\section{Even L\"uroth quartics}\label{Sec:Examples}

Associating a plane quartic curve to a defining equation identifies the set of plane
quartics with $\bP^{14}$.
This projective space has an interesting L\"uroth hypersurface whose general point is a
\demph{L\"uroth quartic}, which is a quartic that contains the ten vertices of some pentalateral
(arrangement of five lines).
The equation for this hypersurface is the \demph{L\"uroth invariant}, which has degree
54~\cite{Mor19} and is
invariant under the induced action of $PGL(3)$ on $\bP(S^4 \bC^3)\simeq\bP^{14}$.
A discussion of this remarkable hypersurface, with references, is given
in~\cite[Remark 6.3.31]{Dolgachev}.

We use the algorithm of Section~\ref{Sec:NPwitness} to investigate the \demph{L\"uroth polytope},
the Newton polytope of the L\"uroth invariant.
While we are not yet able to compute the full L\"uroth polytope, we can compute some of
its vertices, including all those on a particular three-dimensional face.
This face is the Newton polytope of the L\"uroth hypersurface in the five-dimensional
family of \demph{even quartics} whose monomials are squares,
 \begin{equation*}\label{Eq:even}
   \sE\ :=\
    \{ q_{400}x^4+q_{040}y^4+q_{004}z^4+2q_{220}x^2y^2 + 2q_{202}x^2z^2+2q_{022}y^2z^2\::\:
     [q_{400},\dotsc,q_{022}]\in\bP^5\}\,.
 \end{equation*}
(Note the coefficients of 2 on the last three terms.
 This scaling tempers the coefficients in the equation $f_5$ in Figure~\ref{F:polys} for the
 even L\"uroth quartics.)
We show that this Newton polytope is a bipyramid that is affinely isomorphic to
 \begin{equation}\label{Eq:bipyramid}
   \conv\left\{\three{0}{0}{0},\three{1}{0}{0},\three{0}{1}{0},\three{0}{0}{1},
                \three{1}{1}{1}\right\}
       \quad = \quad
   \raisebox{-27pt}{\includegraphics[height=60pt]{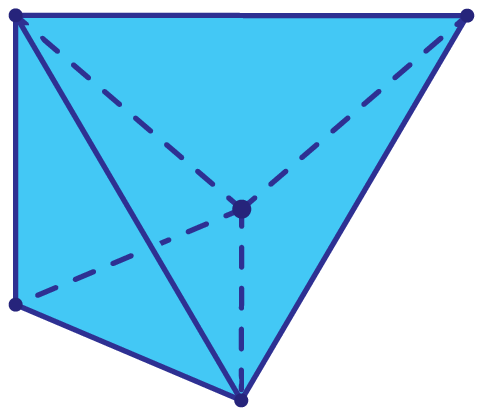}}
 \end{equation}
We will furthermore use the numerical interpolation method of \cite{BHMPS} to compute
the equation for the hypersurface in $\sE$ of even L\"uroth quartics.\smallskip

If  $\ell_1,\dotsc,\ell_5$ are general linear forms on $\bP^2$, then the quartic with
equation
 \begin{equation}\label{Eq:Lueroth}
    \ell_1\ell_2\ell_3\ell_4\ell_5\cdot
    (\tfrac{1}{\ell_1}+\tfrac{1}{\ell_2}+\tfrac{1}{\ell_3}
    +\tfrac{1}{\ell_4}+\tfrac{1}{\ell_5})\ =\ 0\,
 \end{equation}
contains the ten points of pairwise intersection of the five lines defined by
$\ell_1,\dotsc,\ell_5$.
Counting constants suggests that there is a 14-dimensional family of such quartics, but
L\"uroth showed~\cite{Lu69} that the set of such quartics forms a hypersurface in $\bP^{14}$.

The formula~\eqref{Eq:Lueroth} exhibits the \demph{L\"uroth hypersurface $\sLH$} as the
closure of the intersection of a general affine hyperplane $\sM\subset\bC^{15}$ with the
image of the map
 \begin{equation}\label{Eq:g}
  \begin{array}{rcl}
     g\ \colon\ (\bC^3)^5 &\longrightarrow&\bC^{15}\\ \rule{0pt}{15pt}
     (\ell_1,\dotsc,\ell_5)&\longmapsto&{\displaystyle \prod_{i=1}^5\ell_i \cdot \sum_{i=1}^5 \tfrac{1}{\ell_i}}
  \end{array}
 \end{equation}

%
The codimension of $\sLH$ and the dimension of the general fiber (both $1$)
are easily verified using this parameterization \cite[Lemma 3]{WitnessProj}.
In particular, we used the method of~\cite{WitnessProj} described in \S~\ref{Sec:WitnessSets}
with Bertini \cite{BHSW06} to compute a witness set for $\sLH=\overline{g(\bC^{15})\cap\sM}$.
This witness set verifies that the degree of $\sLH$ is $54$.
As shown in \cite{WitnessMemb}, this witness set
also provides the ability to test membership in $\sLH$ by tracking at most $54$ paths.

The space $\bC^{15}$ of quartic polynomials has coordinates given by the coefficients of the
monomials in a quartic,
 \begin{eqnarray*}
   \begin{picture}(80,10)\put(0,0){${\displaystyle\sum_{i+j+k=4} q_{ijk}x^iy^jz^k}$}\end{picture}
    \ &=& q_{400}x^4+q_{310}x^3y+q_{301}x^3z+q_{220}x^2y^2+q_{211}x^2yz\\
       &&+q_{202}x^2z^2+q_{130}xy^3+q_{121}xy^2z+q_{112}xyz^2+q_{103}xz^3\\
    &&+q_{040}y^2+q_{031}y^3z+q_{022}y^2z^2+q_{013}yz^3+q_{004}z^4\,.
 \end{eqnarray*}
In theory, we may use the algorithm of Remark~\ref{Rem:witness} to determine the Newton
polytope of~$\sLH$.  While difficult in practice, we may compute some vertices.
For example,
\[
   q_{400}^6q_{301}^6q_{121}^{30}q_{013}^{12}\ \leftrightarrow\
   (6,0,6,0,0,0,0,30,0,0,0,0,0,12,0)\,,
\]
is the extreme monomial in the direction
\[
  (3,-5,3,2,3,-2,-1,4,-3,-2,3,1,-5,3,-5)\,.
\]
By symmetry, this gives five other vertices,
\[
   q_{400}^6q_{310}^6q_{112}^{30}q_{031}^{12}\,,\,
   q_{040}^6q_{031}^6q_{211}^{30}q_{103}^{12}\,,\,
   q_{040}^6q_{130}^6q_{112}^{30}q_{301}^{12}\,,\,
   q_{004}^6q_{013}^6q_{211}^{30}q_{130}^{12}\,,\,
   q_{004}^6q_{103}^6q_{121}^{30}q_{310}^{12}\,.
\]

It is dramatically more feasible to compute the Newton polytope of the
hypersurface of L\"uroth quartics in the space $\sE$ of even quartics.
This is the face of the L\"uroth polytope that is extreme in the direction of
$v$, where
\[
   v\cdot(q_{400},q_{310},\dotsc,q_{004})\ =\
   -\sum\{ q_{ijk}\mid \mbox{one of $i$, $j$, $k$ is odd}\}\,.
\]
Obtaining a witness set for the even L\"uroth quartics, $\defcolor{\sEH}:=\sE\cap\sLH$,  is
straightforward; we reparameterize using the $2$'s in the definition of $\sE$ and
include the linear equations
\[
   q_{ijk}\ =\ 0\qquad\mbox{where one of $i$, $j$, $k$ is odd}
\]
among the affine linear equations $\sL\colon\bC^{15}\to\bC^{13}$ used for the witness set
computation.  When performing this specialization, some of the 54 points from $\sLH$
coalesce.  More precisely, six points of $\sEH$
arise as the coalescence of four points each, nine points of $\sEH$ arise
the coalescence of two points each, and the remaining twelve points
remain distinct.  This implies that $\sEH$ is reducible with non-reduced components.

Numerical irreducible decomposition shows that $\sEH$ consists of eight components, only one
of which is reduced.
However, as we are using witness sets for images of maps~\cite{WitnessProj} (as described
in \S~\ref{Sec:WitnessSets}) the numerical computations are not performed in $\sE$, but
rather on the smooth incidence variety of the map $g$~\eqref{Eq:g}.

We first determine the Newton polytope of each component and then
use interpolation~\cite{BHMPS} to recover the defining
equation for each component.
For $f_1,\dotsc,f_5$ as given in Figure~\ref{F:polys}, $\sEH$ is defined~by
 \begin{equation}\label{E:EH}
    q_{400}^4\cdot q_{040}^4\cdot q_{004}^4\cdot f_1^4\cdot f_2^2\cdot f_3^2\cdot f_4^2\cdot f_5 \ = 0\,.
 \end{equation}
%
\begin{figure}[htb]
 \begin{eqnarray*}
f_1 &=& q_{400} q_{040} q_{004} - q_{400} q_{022}^2 - q_{040} q_{202}^2 -  q_{004} q_{220}^2 - 2 q_{220} q_{202} q_{022} \\
f_2 &=& q_{400} q_{040} q_{004} - q_{400} q_{022}^2 + 3 q_{040} q_{202}^2 - q_{004} q_{220}^2 + 2 q_{220} q_{202} q_{022} \\
f_3 &=& q_{400} q_{040} q_{004} + 3 q_{400} q_{022}^2 - q_{040} q_{202}^2 - q_{004} q_{220}^2 + 2 q_{220} q_{202} q_{022} \\
f_4 &=& q_{400} q_{040} q_{004} - q_{400} q_{022}^2 - q_{040} q_{202}^2 + 3 q_{004} q_{220}^2 + 2 q_{220} q_{202} q_{022} \\
f_5 &=&
2401 q_{400}^4 q_{040}^4 q_{004}^4
-196 q_{400}^4 q_{040}^3 q_{004}^3 q_{022}^2
+102 q_{400}^4 q_{040}^2 q_{004}^2 q_{022}^4
-4 q_{400}^4 q_{040} q_{004} q_{022}^6\\
&&
+~q_{400}^4 q_{022}^8
-196 q_{400}^3 q_{040}^4 q_{004}^3 q_{202}^2
-196 q_{400}^3 q_{040}^3 q_{004}^4 q_{220}^2
+840 q_{400}^3 q_{040}^3 q_{004}^3 q_{220} q_{202} q_{022}\\
&&
-~820 q_{400}^3 q_{040}^3 q_{004}^2 q_{202}^2 q_{022}^2
-820 q_{400}^3 q_{040}^2 q_{004}^3 q_{220}^2 q_{022}^2
+232 q_{400}^3 q_{040}^2 q_{004}^2 q_{220} q_{202} q_{022}^3\\
&&
-~12 q_{400}^3 q_{040}^2 q_{004} q_{202}^2 q_{022}^4
-12 q_{400}^3 q_{040} q_{004}^2 q_{220}^2 q_{022}^4
-40 q_{400}^3 q_{040} q_{004} q_{220} q_{202} q_{022}^5\\
&&
+~4 q_{400}^3 q_{040} q_{202}^2 q_{022}^6
+4 q_{400}^3 q_{004} q_{220}^2 q_{022}^6
-8 q_{400}^3 q_{220} q_{202} q_{022}^7
+102 q_{400}^2 q_{040}^4 q_{004}^2 q_{202}^4\\
&&
-~820 q_{400}^2 q_{040}^3 q_{004}^3 q_{220}^2 q_{202}^2
+232 q_{400}^2 q_{040}^3 q_{004}^2 q_{220} q_{202}^3 q_{022}
-12 q_{400}^2 q_{040}^3 q_{004} q_{202}^4 q_{022}^2\\
&&
+~102 q_{400}^2 q_{040}^2 q_{004}^4 q_{220}^4
+232 q_{400}^2 q_{040}^2 q_{004}^3 q_{220}^3 q_{202} q_{022}
+128 q_{400}^2 q_{040}^2 q_{004}^2 q_{220}^2 q_{202}^2 q_{022}^2\\
&&
-~80 q_{400}^2 q_{040}^2 q_{004} q_{220} q_{202}^3 q_{022}^3
+6 q_{400}^2 q_{040}^2 q_{202}^4 q_{022}^4
-12 q_{400}^2 q_{040} q_{004}^3 q_{220}^4 q_{022}^2\\
&&
-~80 q_{400}^2 q_{040} q_{004}^2 q_{220}^3 q_{202} q_{022}^3
+220 q_{400}^2 q_{040} q_{004} q_{220}^2 q_{202}^2 q_{022}^4
-24 q_{400}^2 q_{040} q_{220} q_{202}^3 q_{022}^5\\
&&
+~6 q_{400}^2 q_{004}^2 q_{220}^4 q_{022}^4
-24 q_{400}^2 q_{004} q_{220}^3 q_{202} q_{022}^5
+24 q_{400}^2 q_{220}^2 q_{202}^2 q_{022}^6\\
&&
-~4 q_{400} q_{040}^4 q_{004} q_{202}^6
-12 q_{400} q_{040}^3 q_{004}^2 q_{220}^2 q_{202}^4
-40 q_{400} q_{040}^3 q_{004} q_{220} q_{202}^5 q_{022}\\
&&
+~4 q_{400} q_{040}^3 q_{202}^6 q_{022}^2
-12 q_{400} q_{040}^2 q_{004}^3 q_{220}^4 q_{202}^2
-80 q_{400} q_{040}^2 q_{004}^2 q_{220}^3 q_{202}^3 q_{022}\\
&&
+~220 q_{400} q_{040}^2 q_{004} q_{220}^2 q_{202}^4 q_{022}^2
-24 q_{400} q_{040}^2 q_{220} q_{202}^5 q_{022}^3
-4 q_{400} q_{040} q_{004}^4 q_{220}^6\\
&&
-~40 q_{400} q_{040} q_{004}^3 q_{220}^5 q_{202} q_{022}
+220 q_{400} q_{040} q_{004}^2 q_{220}^4 q_{202}^2 q_{022}^2
-272 q_{400} q_{040} q_{004} q_{220}^3 q_{202}^3 q_{022}^3\\
&&
+~48 q_{400} q_{040} q_{220}^2 q_{202}^4 q_{022}^4
+4 q_{400} q_{004}^3 q_{220}^6 q_{022}^2
-24 q_{400} q_{004}^2 q_{220}^5 q_{202} q_{022}^3\\
&&
+~48 q_{400} q_{004} q_{220}^4 q_{202}^2 q_{022}^4
-32 q_{400} q_{220}^3 q_{202}^3 q_{022}^5
+q_{040}^4 q_{202}^8
+4 q_{040}^3 q_{004} q_{220}^2 q_{202}^6\\
&&
-~8 q_{040}^3 q_{220} q_{202}^7 q_{022}
+6 q_{040}^2 q_{004}^2 q_{220}^4 q_{202}^4
-24 q_{040}^2 q_{004} q_{220}^3 q_{202}^5 q_{022}
+24 q_{040}^2 q_{220}^2 q_{202}^6 q_{022}^2\\
&&
+~4 q_{040} q_{004}^3 q_{220}^6 q_{202}^2
-24 q_{040} q_{004}^2 q_{220}^5 q_{202}^3 q_{022}
+48 q_{040} q_{004} q_{220}^4 q_{202}^4 q_{022}^2\\
&&
-32~q_{040} q_{220}^3 q_{202}^5 q_{022}^3
+q_{004}^4 q_{220}^8
-8 q_{004}^3 q_{220}^7 q_{202} q_{022}
+24 q_{004}^2 q_{220}^6 q_{202}^2 q_{022}^2\\
&&
-~32 q_{004} q_{220}^5 q_{202}^3 q_{022}^3
+16 q_{220}^4 q_{202}^4 q_{022}^4\,.\vspace{-5pt}
\end{eqnarray*}
\caption{Polynomials defining $\sEH$}\label{F:polys}
\end{figure}
For completeness, we used the algorithm of~\cite{WitnessMemb} to verify that a random
element of each hypersurface $\Var(f_i)$ lies on~$\sLH$.

Observe that $f_1,f_2,f_3$, and $f_4$ all have the same support and therefore the same Newton
polytope,~$\Delta$.   Every integer point of $\Delta$ corresponds to a monomial
in these polynomials and all are extreme.
The Newton polytope of $f_5$ is $4\Delta$ and it has 65 nonzero terms, which correspond to
all the integer points in $4\Delta$.
Thus the Newton polytope of $\sEH$ is $14\Delta+\alpha$, where $\alpha$ is the exponent
vector of $q_{400}^4q_{040}^4q_{004}^4$.
To complete the identification of $\sN(\sEH)$,
consider the integer points $\{O,A,B,C,D\}$ of $\Delta$, which are on the left in
Table~\ref{T:one}
\begin{table}[htb]
\caption{Vertices of $\Delta$}\label{T:one}
$
   \begin{array}{ccccccc}
     &q_{400}&q_{040}&q_{004}&q_{022}&q_{202}&q_{220}\\
    O&0&0&0&1&1&1\\
    A&1&0&0&2&0&0\\
    B&0&1&0&0&2&0\\
    C&0&0&1&0&0&2\\
    D&1&1&1&0&0&0\end{array}
   \qquad
   \begin{array}{rrrrrrr}
     &q_{400}&q_{040}&q_{004}&q_{022}&q_{202}&q_{220}\\
    o&0&0&0&0&0&0\\
    a&1&0&0&1&-1&-1\\
    b&0&1&0&-1&1&-1\\
    c&0&0&1&-1&-1&1\\
    d&1&1&1&-1&-1&-1\end{array}
$
\end{table}
Replacing $\{O,\dotsc,D\}$ by their differences with $O$ gives the points $o,a,b,c,d$ on
the right in Table~\ref{T:one}.
Note that $a+b+c=d$.  Projecting to the first three coordinates is an isomorphism of the
integer span of $a,b,c$ with $\bZ^3$, and shows that $\Delta$ is affinely isomorphic to
the bipyramid~\eqref{Eq:bipyramid}

Using \eqref{E:EH}, we can determine which Edge quartics \cite{Edge38,PSV11} are L\"uroth quartics
since the family of Edge quartics $\sED$ is contained in $\sE$ with
 \[
    \sED\ :=\ \{\Var(s(x^4+y^4+z^4)-t(y^2z^2+x^2z^2+x^2y^2)) \::\: [s,t]\in\bP^1\}\,.
 \]
Identifying $\sED$ with $\bP^1$, and evaluating at~\eqref{E:EH} gives the equation for
$\sED\cap\sLH$,
 \[
    s^{12}(s+t)^4(2s-t)^{16}(7s+t)(2s^2+st+t^2)^6(28s^3+8s^2t+3st^2+t^3)^3\ =\ 0\,.
%
%
 \]
%
%
Set $\omega := \sqrt[3]{297 + 24 \sqrt{159}}$.
Besides the point $[0,1]$, the eight points $[1,t]$ corresponding to
Edge quartics that are L\"uroth quartics are
\begin{eqnarray*}
t_1 &=& -1\,,\\
t_2 &=& 2\,,\\
t_3 &=& -7\,,\\
t_4 &=& \frac{1}{2}(\sqrt{-7} - 1)\,,\\
t_5 &=& \frac{-1}{2}(1 + \sqrt{-7})\,,\\
t_6 &=& \frac{1}{3\omega}(15 - 3 \omega - \omega^2)\,, \\
t_7 &=& \frac{1}{6\omega}(\omega^2 - 6\omega - 15 + \sqrt{-3}~(\omega^2 + 15))\,,\\
t_8 &=& \frac{1}{6\omega}(\omega^2 - 6\omega - 15 - \sqrt{-3}~(\omega^2 + 15))\,.
\end{eqnarray*}

In particular, there are four real values $t_1,t_2,t_3,t_6$ and four nonreal values $t_4,t_5,t_7,t_8$.
The Edge L\"uroth quartic corresponding to $[0,1]$ has three real points, each of which is singular.
Also, except for $t=t_2=2$, which is the union of four lines
\[
    x - y + z\ =\ x - y - z \ =\ x + y - z\ =\ x + y + z\ =\ 0\,,
\]
the Edge L\"uroth quartic corresponding to $[1,t_i]$ is smooth with no real points.


%
\section{Conclusion}\label{Sec:Conclusion}
We presented two algorithms for computing the Newton polytope of a hypersurface $\sH$
given numerically.
The first assumes that we may evaluate a polynomial defining $\sH$ while the second uses a
witness set representation of $\sH$.
The second is illustrated through the determination of the polynomial defining the
hypersurface of even L\"uroth quartics (which gives a face of the L\"uroth poytope), along
with some other vertices of the L\"uroth polytope.
Implementing these algorithms remains a future project.

%
\section*{Acknowledgments} The authors would like to thank
Institut Mittag-Leffler (Djursholm, Sweden) for support and hospitality,
and Bernd Sturmfels for his questions during this program regarding the
numerical computation of Newton polytopes.


\bibliographystyle{amsplain}
\bibliography{bibl}

\providecommand{\bysame}{\leavevmode\hbox to3em{\hrulefill}\thinspace}
\providecommand{\MR}{\relax\ifhmode\unskip\space\fi MR }
\providecommand{\MRhref}[2]{%
  \href{http://www.ams.org/mathscinet-getitem?mr=#1}{#2}
}
\providecommand{\href}[2]{#2}
\begin{thebibliography}{10}

\bibitem{BHSW06}
D.J. Bates, J.D. Hauenstein, A.J. Sommese, and C.W. Wampler, \emph{Bertini:
  Software for numerical algebraic geometry}, Available at
  \url{http://www.nd.edu/\~sommese/bertini}.

\bibitem{BHMPS}
D.J. Bates, J.D. Hauenstein, McCoy T.M., C.~Peterson, and A.J. Sommese,
  \emph{Recovering exact results from inexact numerical data in algebraic
  geometry}, Exp. Math., to appear.

\bibitem{Berg}
G.M. Bergman, \emph{The logarithmic limit-set of an algebraic variety}, Trans.
  Amer. Math. Soc. \textbf{157} (1971), 459--469.

\bibitem{BiGr}
R.~Bieri and J.R.J. Groves, \emph{The geometry of the set of characters induced
  by valuations}, J. Reine Angew. Math. \textbf{347} (1984), 168--195.

\bibitem{CK70}
D.R. Chand and S.S. Kapur, \emph{An algorithm for convex polytopes}, J. Assoc.
  Comput. Mach. \textbf{17} (1970), 78--86.

\bibitem{CGKW}
R.~Corless, M.~Giesbrecht, I.~Kotsireas, and S.~Watt, \emph{Numerical
  implicitization of parametric hypersurfaces with linear algebra}, Artificial
  Intelligence and Symbolic Computation, Lecture Notes in Computer Science,
  vol. 1930, Springer-Verlag, 2000, pp.~174--183.

\bibitem{CLO}
D.~Cox, J.~Little, and D.~O'Shea, \emph{Ideals, varieties, and algorithms},
  third ed., Undergraduate Texts in Mathematics, Springer, New York, 2007.

\bibitem{Dolgachev}
I.V. Dolgachev, \emph{Classical algebraic geometry: a modern view}, Cambridge
  Univ. Press, 2012, 656 pp.

\bibitem{Edge38}
W.L. Edge, \emph{{Determinantal representations of $x^4+y^4+z^4$}}, Math. Proc.
  Cambridge Phil. Soc. \textbf{34} (193), 6--21.

\bibitem{IKL10}
I.Z. Emiris, C.~Konaxis, and L.~Palios, \emph{Computing the {N}ewton polygon of
  the implicit equation}, Math. Comput. Sci. \textbf{4} (2010), no.~1, 25--44.

\bibitem{EK05}
I.Z. Emiris and I.~Kotsireas, \emph{Implicitization exploiting sparseness},
  Geometric And Algorithmic Aspects Of Computer-aided Design And Manufacturing
  (D.~Dutta, M.~Smid, and R.~Janardan, eds.), DIMACS Series in Discrete
  Mathematics and Theoretical Computer Science, vol.~67, American Mathematical
  Society, Providence RI, 2005, pp.~281--298.

\bibitem{ES10}
A.~Esterov, \emph{Newton polyhedra of discriminants of projections}, Discrete
  Comput. Geom. \textbf{44} (2010), no.~1, 96--148.

\bibitem{EK08}
A.~Esterov and A.~Khovanskii, \emph{Elimination theory and {N}ewton polytopes},
  Funct. Anal. Other Math. \textbf{2} (2008), no.~1, 45--71.

\bibitem{Gru03}
B.~Gr{\"u}nbaum, \emph{Convex polytopes}, second ed., Graduate Texts in
  Mathematics, vol. 221, Springer-Verlag, New York, 2003, Prepared and with a
  preface by Volker Kaibel, Victor Klee and G{\"u}nter M. Ziegler.

\bibitem{WitnessProj}
J.D. Hauenstein and A.J. Sommese, \emph{Witness sets of projections}, Appl.
  Math. Comput. \textbf{217} (2010), no.~7, 3349--3354.

\bibitem{WitnessMemb}
\bysame, \emph{Membership tests for images of algebraic sets by linear
  projections}, Preprint available at
  \url{www.math.ncsu.edu/~jdhauens/preprints}, 2012.

\bibitem{Isosingular}
J.D. Hauenstein and C.W. Wampler, \emph{Isosingular sets and deflation},
  Preprint available at \url{www.math.ncsu.edu/~jdhauens/preprints}, 2011.

\bibitem{iB4e}
P.~Huggins, \emph{i{B}4e: a software framework for parametrizing specialized
  {LP} problems}, Mathematical software---{ICMS} 2006, Lecture Notes in Comput.
  Sci., vol. 4151, Springer, Berlin, 2006, pp.~245--247.

\bibitem{JY}
A.~Jensen and J.~Yu, \emph{Computing tropical resultants}, {\tt
  arXiv:1109.2368}.

\bibitem{Lu69}
J.~L{\"u}roth, \emph{Einige {E}igenschaften einer gewissen {G}attung von
  {C}urven vierten {O}rdnung}, Math. Ann \textbf{1} (1869), 37--53.

\bibitem{Mor19}
F.~Morley, \emph{On the {L}{\"u}roth quartic curve}, Amer. J. Math \textbf{41}
  (1919), 279--282.

\bibitem{OWM83}
T.~Ojika, S.~Watanabe, and T.~Mitsui, \emph{Deflation algorithm for the
  multiple roots of a system of nonlinear equations}, J. Math. Anal. Appl.
  \textbf{96} (1983), no.~2, 463--479.

\bibitem{PSV11}
D.~Plaumann, B.~Sturmfels, and C.~Vinzant, \emph{Quartic curves and their
  bitangents}, J. Symbolic Comput. \textbf{46} (2011), no.~6, 712--733.

\bibitem{SW05}
A.J. Sommese and C.W. Wampler, II, \emph{The numerical solution of systems of
  polynomials arising in engineering and science}, World Scientific Publishing
  Co. Pte. Ltd., Hackensack, NJ, 2005.

\bibitem{ST08}
B.~Sturmfels and J.~Tevelev, \emph{Elimination theory for tropical varieties},
  Math. Res. Lett. \textbf{15} (2008), no.~3, 543--562.

\bibitem{STY}
B.~Sturmfels, J.~Tevelev, and J.~Yu, \emph{The {N}ewton polytope of the
  implicit equation}, Mosc. Math. J. \textbf{7} (2007), no.~2, 327--346, 351.

\bibitem{SY}
B.~Sturmfels and J.~Yu, \emph{Tropical implicitization and mixed fiber
  polytopese}, Software for Algebraic Geometry (Michael Stillman, Jan
  Verschelde, and Nobuki Takayama, eds.), IMA Volumes in Mathematics and its
  Applications, vol. 148, Springer, 2008, pp.~111--131.

\end{thebibliography}

\end{document}